\documentclass[final,twosided]{siamltex}
\usepackage[pdftex]{hyperref}
\usepackage{amsmath,amsfonts}

\newcommand{\mbe}{\mathbb{E}}

\newcommand{\mbp}{\mathbb{P}}

\newcommand{\cN}{{\mathcal{N}}}
\newcommand{\opA}{A}
\newcommand{\nat}{{\mathbb N}}
\newcommand{\real}{\mathbb R}

\renewcommand{\rho}{{\varrho}}
\newcommand{\rhon}{\rho_{\cN}}

\newcommand{\lr}[1]{\left(#1\right)}
\newcommand{\abs}[1]{\left\vert #1 \right\vert}
\newcommand{\norm}[2]{\Vert #1 \Vert _{#2}}
\newcommand{\set}[1]{\left\{#1\right\}}
\newcommand{\scalar}[2]{\langle #1,#2\rangle}
\newcommand{\brac}[1]{\left[#1\right]}
\newcommand{\ra}{r_\alpha}
\newcommand{\sa}{s_\alpha}
\newcommand{\ga}{g_\alpha}
\newcommand{\xa}{x_\alpha}
\newcommand{\xn}{x_{0}}

\newcommand{\rb}{r_\beta}
\newcommand{\raa}{r_{\alpha_\ast}}
\newcommand{\saa}{s_{\alpha_\ast}}
\newcommand{\sbeta}{s_\beta}
\newcommand{\xb}{x_\beta}
\newcommand{\xad}{x_\alpha^\delta}
\newcommand{\xast}{x^{\dagger}}
\newcommand{\aast}{\alpha_\ast}

\newcommand{\yd}{y^\delta}
\newcommand{\zd}{z^\delta}

\def\a{\alpha}

\newcommand{\tr}[1]{\mathrm{Tr}\brac{#1}}
\newcommand{\e}[1]{\mbe\brac{#1}}

\newcommand{\prob}[1]{\mbp\brac{#1}}

\newenvironment{frcseries}{\fontfamily{frc}\selectfont}{}

\newtheorem{thm}{Theorem}[section]

\newtheorem{cor}{Corollary}[section]
\newtheorem{lem}{Lemma}[section]
\newtheorem{rem}{Remark}[section]
\newtheorem{de}{Definition}[section]
\newtheorem{ass}{Assumption}[section]
\newtheorem{example}{Example}[section]

\title{Oracle inequality for
 a statistical Raus--Gfrerer type rule }
\author{Qinian Jin\thanks{Mathematical Sciences Institute, Bldg 27, Australian National
University, ACT 0200 Australia} ({\tt qinian.jin@anu.edu.au})
\and Peter Math\'e\thanks{Weierstra{\ss} Institute for Applied Analysis and
  Stochastics, Mohrenstra{\ss}e 39, 10117 Berlin,  Germany} ({\tt peter.mathe@wias-berlin.de})}

\begin{document}

\maketitle
\begin{abstract}
  The authors study statistical linear inverse problems in Hilbert
  spaces. Approximate solutions are sought within a class of linear
  one-parameter regularization schemes, and the parameter choice is
  crucial to control the root mean squared error. Here a variant of
  the Raus--Gfrerer rule is analyzed, and it is shown that this
  parameter choice gives rise to error bounds in terms of oracle
  inequalities, which in turn provide order optimal error bounds (up
  to logarithmic factors). These bounds can only be
  established for solutions which obey a certain self-similarity
  structure. The proof of the main result relies on some auxiliary
  error analysis for linear inverse problems under general noise
  assumptions, and this may be interesting in its own.
\end{abstract}
\begin{keywords}
  statistical inverse problem, Raus--Gfrerer parameter choice,
oracle inequality
\end{keywords}
\begin{AMS}
  47A52, secondary: 65F22, 65C60
\end{AMS}
\pagestyle{myheadings}
\thispagestyle{plain}
\markboth{Q. JIN AND P. MATH\'E}{RG RULE FOR STATISTICAL INVERSE PROBLEMS}

\section{Introduction}

\label{sec:intro}

In this study we introduce a new parameter choice strategy for
statistical linear
inverse problems in Hilbert spaces. We consider the
following linear equation
\begin{equation}
  \label{eq:model}
  \yd = T \xast + \delta \xi,
\end{equation}
where $T\colon X \to Y$ is a compact linear operator between
Hilbert spaces $X$ and $Y$, the parameter $\delta>0$ denotes the noise
level, and $\xi$ stands for the additive noise, to be specified later
as Gaussian white noise,  which leads to observations $\yd$. This is a
standard model considered in statistical inverse problems. By using
the singular system $\{s_j, u_j, v_j\}$ of $T$ to write $T x =
\sum_j s_{j}\scalar{x}{u_{j}}v_{j},\ x\in X$, the above model~(\ref{eq:model}) is
seen to be equivalent to the sequence space model
$$
y_{j}^{\delta} = x_{j} + \delta \xi_{j},\quad j=1,2,\dots,
$$
with observations~$y_{j}^{\delta}= \scalar{\yd}{v_{j}}/s_{j}$, the
noise~$\xi_{j}$ is centered Gaussian with
variance~$\delta^{2}/s_{j}^{2}$. The unknown solution~$x$ has coefficients~$x_{j}$
with respect to the basis~$u_{j},\ j=1,2,\dots$
This model is frequently analyzed, and we mention the recent
survey~\cite{MR2421941}. In particular the minimax error is clearly
understood if the solution sequence~$x_{j},\ j=1,2,\dots$ belongs to
some Sobolev type ball. In particular, a series estimator $\hat x_{k}(\yd) =
\sum_{j=1}^{k} c_{j} \yd_{j}$ (with appropriately chosen weights
$c_{j}$) is (almost) optimal.

The important question is how to choose the truncation level
(parameter, model)~$k$ based on the given data and the noise level~$\delta$.
Parameter choice in statistical inverse problems,
called \emph{model selection} in this field, is an important issue,
and we refer to~\cite{MR2421941} for a survey on this. Only recently, the
\emph{discrepancy principle}, which is the most prominent parameter choice in classical
regularization theory, has been analyzed within the statistical context in~\cite{B/M2010}.
Here, for any estimator $\hat x= \hat x(\yd)$ it requires to
achieve that $\norm{T \hat x-\yd}{}\asymp \delta$.  Since the white noise $\xi$ is
not an element in $Y$, the discrepancy $\|T \hat{x} -y^\delta\|$ is not well-defined.
Therefore, for statistical inverse problems, the traditional discrepancy principle can not
be applied directly.

In order to make the discrepancy principle applicable to statistical inverse problems,
we may consider, instead, the symmetrized equation with $A:= T^{\ast}T\geq 0$ and
$\zeta:=T^{\ast} \xi$, as
\begin{equation}
  \zd = T^{\ast} \yd = \opA \xast + \delta T^{\ast} \xi = \opA \xast + \delta \zeta.
\label{eq:model-symm}
\end{equation}
Then, if the operator~$\opA$ has finite trace, the new misfit $\norm{\opA \hat x -
  \zd}{}$ is almost surely finite, and it is tempting to require that
\begin{equation}
  \label{eq:DP-plain}
  \norm{\opA \hat x(\zd) - \zd}{}\asymp \delta,
\end{equation}
which gives the discrepancy
principle for the symmetrized equation. However, as was pointed out
in~\cite{B/M2010}, this plain use of the discrepancy principle leads only
to suboptimal performance. Instead, the misfit~$\opA \hat x(\zd)
- \zd$ should be weighted, and if done accordingly, this can yield
optimal rates of reconstruction. To be specific we consider the family of reconstructions
$$
\xad = \lr{\alpha I +
  \opA}^{-1}T^{\ast}\yd,\qquad \alpha>0,
$$
via Tikhonov regularization. The authors in \cite{B/M2010} studied the
\emph{modified  discrepancy principle}
 \begin{equation}
   \label{eq:mdp}
 \norm{\lr{\lambda I + \opA}^{ -1/2}(A \xad
   - \zd)}{}\asymp \delta.
 \end{equation}
It is shown that an appropriate choice of $\lambda>0$ yields order
optimal reconstruction in many cases. However, the choice of $\lambda$ requires
the unknown smoothness of solution which makes the discrepancy principle into
an \emph{a priori} rule.

Instead, the authors in \cite{Lu/Mathe2012} considered the \emph{varying discrepancy principle}
\begin{equation}
  \label{eq:vdp}
  \norm{\lr{\alpha I + \opA}^{-1/2}(A \xad - \zd)}{}\asymp \delta
\end{equation}
by relating $\lambda=\alpha$ in (\ref{eq:mdp}) to make the principle into an
\emph{a posteriori} one, and thus the weight depends on the parameter
$\alpha$ under consideration. The main achievement in \cite{Lu/Mathe2012} is
that this new principle may yield optimal order reconstruction (up to a logarithmic
factor). However, it became transparent that such result holds only
for solutions $\xast$ which satisfy certain self-similarity
properties. This has an intuitive explanation: For large values of
$\alpha$, and this is where the discrepancy principle starts with, the
misfit is dominated by the large singular numbers~$s_{j}$. However,
the approximation order is determined by all of the spectrum.

The varying discrepancy principle has another drawback.
The regularization scheme, which is used to
determine the candidate solutions $\xad$ must have higher
qualification than given by the underlying smoothness in terms of
general source conditions. For instance, if we use Tikhonov
regularization, whose qualification is known
to be 1, see \cite{EHN96}, then the varying discrepancy principle gives order optimal
reconstruction only for smoothness `up to 1/2'.  This effect, which
is inherent in the discrepancy principle in classical regularization
context, is called \emph{early saturation}, and it can be overcome
by turning from the discrepancy principle to the so-called
\emph{Raus--Gfrerer rule} (RG-rule).

As Raus and Gfrerer proposed, instead of the discrepancy
from~(\ref{eq:DP-plain}) an additional weight should be used,
which results in the RG-rule
\begin{equation}
  \label{eq:RG-plain}
  \norm{\lr{\alpha I + A}^{-1}(\opA \xad - \zd)}{}\asymp \delta.
\end{equation}
This is the starting point for the present study, the application of
the RG-rule within the statistical context. It will be shown that an
appropriate use of the RG-rule will yield order optimal results
without the effect of early saturation. Actually, we will propose a
statistical version of RG rule and establish some \emph{oracle inequalities},
provided that the solution obeys some self-similarity. Oracle inequalities
are widely used in statistics, see \cite{MR2421941}. An oracle inequality
guarantees that the estimator has a risk of the same order as that of the oracle.
The oracle bound in particular implies that Tikhonov regularization
can achieve order optimal reconstruction up to order $1$.

This paper is organized as follows. We first precisely introduce the context, and then
we state the main result with some discussion in Section \ref{sec:setup}. The proof of
the main result will rely on preliminary results within the classical
(deterministic noise) setting given in Section \ref{sec:aux}, however, under \emph{general noise
assumptions}. The results in this context may be interesting
in their own.  Finally, the proof of the main result is given in
Section \ref{sec:gwn}.

\section{Setup and main result}

\label{sec:setup}

We shall use the same setup as in \cite{B/M2010,Lu/Mathe2012}. However, the parameter
choice will be different.

\subsection{Assumptions}
\label{sec:ass}

We start with the description of the noise.
We will mimic the notion of \emph{Gaussian white noise} to the present case. Let
$\lr{\Omega,\mathcal F, {\mathbb P}}$ be a (complete) probability
space, and let $\mathbb E$ be the expectation with respect to ${\mathbb P}$.

\begin{ass}[Gaussian white noise]\label{ass:GWN}
{\it  The noise $\xi = \lr{\xi(y),\ y\in Y}$ in~(\ref{eq:model})
is a stochastic process, defined on $\lr{\Omega,\mathcal F, {\mathbb P}}$ with
the properties that
\begin{enumerate}
\item for each $y\in Y$ the random number~$ \xi(y)\in
  L_{2}\lr{\Omega,\mathcal F, {\mathbb P}}$ is a centered Gaussian random
  variable, and
\item  for all $y,y^{\prime}\in Y$ the covariance structure
  is~$\e{\xi(y)\xi(y^{\prime})} = \scalar{y}{y^{\prime}}$.
\end{enumerate}
}
\end{ass}

As a consequence, the mapping $y\to \xi(y)$ is linear, and we shall
thus write $\xi(y) = \scalar{\xi}{y}$, we refer to~\cite{MR543837} for
details.

The related Gaussian process~$\zeta:= T^*\xi$ has
covariance $\e{\scalar{\zeta}{w}\scalar{\zeta}{w^{\prime}}} = \scalar{w}{\opA
w^{\prime}}$, $w,w^{\prime}\in X$ with the operator $\opA:= T^{ \ast} T$.

\begin{ass}\label{ass:power-type}
{\it
The operator $A$ has finite trace $\tr A<\infty$.
}
\end{ass}

Under Assumption \ref{ass:power-type}, Sazonov's Theorem,
cf. \cite{MR543837}, asserts that the element $\zeta:=T^* \xi$ is a Guassian random element in
$X$ (almost surely). Therefore the equation
\begin{equation}
  \label{eq:main-model}
\zd = \opA \xast + \delta \zeta
\end{equation}
is a well defined linear equation in $X$ (almost surely). This will be
our main model from now on.

Moreover, Assumption \ref{ass:power-type} implies that the following function
is well defined; for further properties we refer to \cite{B/M2010}.

\begin{de}[effective dimension]\label{def:cn}
The function $\cN(\lambda)$ defined as
\begin{equation}
\label{ass:opreg}
\cN (\lambda) = \cN_{\opA} (\lambda):= \tr{(\opA+\lambda I)^{-1}\opA}
, \, \lambda>0,
\end{equation}
is called effective dimension of the operator $\opA$ under white noise.
\end{de}

Along with the effective dimension, as in~\cite{Lu/Mathe2012} we introduce
the decreasing function~$\rhon(t)$ given by
\begin{equation}
  \label{eq:rhon}
  \rhon(t):= 1/\sqrt{t\cN(t)},\quad t>0
\end{equation}
and its companion
\begin{equation}
  \label{eq:thetaf}
  \Theta_{\rhon}(t) := t \rhon(t),\quad t>0.
\end{equation}
The latter function is continuous and strictly increasing, hence its
inverse is well-defined.

We recall the notion of linear regularization, see
e.g.~\cite[Definition 2.2]{MR2318806}.
\begin{de}[linear regularization]\label{de:regularization}
  A family of functions
$$
\ga\colon (0,\norm{\opA}{}]\mapsto \real,\
  0<\alpha \leq \norm{\opA}{},
$$
is called
  regularization if they are piecewise continuous in~$\alpha$ and the
  following properties hold:
  \begin{enumerate}
  \item\label{it:convergence} For each~$0< t \leq \norm{\opA}{}$ we
    have that~$\abs{\ra(t)}\to 0$ as~$\alpha\to 0$;

  \item \label{it:gamma1} There is a constant~$\gamma_{1}$ such
    that~$\sup_{0\le t\le \|A\|} \abs{\ra(t)}\leq \gamma_{1}$ for all~$0<\alpha \leq \norm{\opA}{}$;

  \item\label{it:gammaast} There is a constant~$\gamma_{\ast}\geq 1$ such
    that~$\sup_{0\le t \leq
      \norm{\opA}{}} \alpha \abs{\ga(t)}\leq \gamma_{\ast}$ for all
    ${0< \alpha < \infty}$,
  \end{enumerate}
where~$\ra(t) := 1 - t \ga(t),\ 0\le t \leq
  \norm{\opA}{},$ denotes the residual function.
\end{de}

We further restrict the analysis to regularization schemes which are
monotone
\begin{align}  \label{eq:ra-monotone}
  \ra(t) &\leq \rb(t),\qquad \text{for } 0 < \alpha \leq \beta
\end{align}
and
\begin{align}\label{eq:ra-0-1}
 0 &\leq \ra(t) \leq 1, \qquad \mbox{for } \a>0.
\end{align}
Hence Item~(\ref{it:gamma1}) in Definition \ref{de:regularization} holds with $\gamma_{1}=1$,
and also $0 \leq t \ga(t)\leq 1$.
We also recall the following fact from~\cite[Lemma 2.3]{MR2684361}: For
$0 < \alpha \leq \beta $ there holds
\begin{equation}
  \label{eq:lem23}
  0 \leq \rb(t) - \ra(t) \leq (1 + \gamma_{\ast})\frac{t}{\alpha+t} \rb(t).
\end{equation}
Indeed, it follows from (\ref{eq:ra-monotone}) and (\ref{eq:ra-0-1}) that
$$
  0 \leq \rb(t) - \ra(t) \leq (1- r_\alpha (t))\rb(t)
= t\ga(t)\rb(t).
$$
The result now follows from the observation that
$(t + \alpha)\ga(t) \leq 1 + \gamma_{\ast}$.

Having chosen an initial guess $x_0\in X$ and a regularization~$\ga$
we construct the approximate solutions
\begin{displaymath}
  \xad  := x_{0} - \ga(\opA)( \opA
x_{0} -\zd),\quad\mbox{and}  \quad \xa := x_{0} - \ga(\opA)(\opA x_{0} -z);
\end{displaymath}
for the noise free case  we use $z:= \opA \xast$.
Recall that the element $\zeta=T^{\ast}\xi$ is a
Gaussian random element in $X$ (almost surely). Therefore, we will use the root
mean squared error at a solution instance $\xast$, given as
\begin{equation}
  \label{eq:rms-error}
\lr{\e{\norm{ \xast - \xad}{}^{2}}}^{1/2},\quad \alpha,\delta>0.
\end{equation}

\subsection{Parameter choice}
\label{sec:VRG}

For the stopping criterion we will consider the following
setup. Having chosen a constant $0< q <1$ we select the parameter
$\alpha$ from the geometric family
\begin{equation}
  \label{eq:Delta}
  \Delta_{q} := \set{\alpha_{k},\ \alpha_{k}:= q^{k}\alpha_{0},\quad k=0,1,2,\dots}.
\end{equation}
For the statistical RG-rule  we introduce the family of functions
\begin{equation}
  \label{eq:sa}
 \sa(t) = \frac{\alpha}{t+\alpha},\quad t,\alpha>0,
\end{equation}
which are the residual functions from Tikhonov regularization.

\begin{de}[statistical RG-rule]\label{de:VRG}
Given $\tau > 1$, $\eta>0$ and $\kappa\geq 0$,  let $\a_{RG}$ be the largest
  parameter~$\alpha\in \Delta_{q}$ for which either
  \begin{align}
\norm{\sa(\opA)(\opA \xad - \zd)}{} & \leq \tau(1+\kappa) \frac
\delta{\rhon(\alpha)},\label{ali:regular}
\end{align}
or
\begin{align}
  \Theta_{\rhon}(\alpha) \leq \eta (1+\kappa)\delta. \label{ali:emergency}
  \end{align}
\end{de}

We will call the criteria (\ref{ali:regular}) and (\ref{ali:emergency}) the regular stop
and emergency stop, respectively. Notice that the regular stop in Definition \ref{de:VRG}
can be viewed as the Raus-Gfrerer rule applied to Lavrent'iev type
regularization of the symmetrized equation~(\ref{eq:main-model}).

\subsection{Restricting the solution set}
\label{sec:solset}

One important observation in the subsequent analysis, in particular
in Section~\ref{sec:aux}, will be that the RG-rule as
introduced in \S\ref{sec:VRG} may fail for statistical problems (and
also for bounded deterministic general noise), if the solution
element~$\xast$ has \emph{abnormal} spectral behavior relative to the
operator~$A$. Therefore, we shall need the following restriction for
the solution $\xast$.
 To describe this we use the spectral resolution~$\lr{E_{t}}_{0 \leq t \leq
\norm{\opA}{}}$ of the (compact) non-negative self-adjoint operator~$\opA$.

\begin{ass}\label{ass:KN}
{\it
There exist $c_1>1$, $0<c_2< 1$ and $0<t_0<\|A\|$ such that
$$
\int_0^\alpha d \|E_t (x^\dag-x_0)\|^2 \le c_1^{2} \int_{c_2 \alpha}^\infty r_\alpha^2(t)
\, d \|E_t(x^\dag-x_0)\|^2
$$
for all $0<\alpha\le t_0$.
}
\end{ass}

The inequality in Assumption \ref{ass:KN} with $c_2=1$ was introduced
in \cite{0266-5611-24-5-055005} as a generalization of a restricted form on
$x^\dag-x_0$ in \cite{MR2395145} for the (iterated) Tikhonov regularization.

\begin{example}
{\rm
 For the $n$-times iterated Tikhonov regularization, we have $\ra(t) =
 \alpha^{n}/(t+\alpha)^{n}$. It is easy to see that
$$
\abs{\ra(t)} \geq c_3 \lr{\frac \alpha t }^{n} \quad \mbox{ for }
t\geq c_2 \alpha,
$$
with $c_3:= (c_2/(1+c_2))^n$. Therefore, in this case, Assumption~\ref{ass:KN} is
equivalent to
$$
\int_{0}^{\alpha}  d \|E_t(\xast - \xn)\|^2 \le c_4 \alpha^{2 n} \int_{c_2\alpha}^{\infty} t^{- 2 n} \;
d \|E_{t}(\xast - \xn)\|^{2},\quad  0< \alpha \leq t_0.
$$
This, with $c_2=1$,  is the condition used in~\cite{MR2395145}.
}
\end{example}

\begin{example}
{\rm
For truncated singular value decomposition method we have
$$
g_\alpha(t) =\left\{\begin{array}{lll}
1/t, & t\ge \alpha,\\
0, & t<\alpha
\end{array}\right.
\quad \mbox{and} \quad
r_\alpha(t) = \left\{ \begin{array}{lll}
0, & t\ge \alpha,\\
1, & t<\alpha.
\end{array}\right.
$$
Thus Assumption \ref{ass:KN} becomes
$$
\int_0^\a d \|E_t (x^\dag-x_0)\|^2 \le c_1 \int_{c_2\a}^\a d \|E_t(x^\dag-x_0)\|^2,
\quad \forall 0<\a \le t_0.
$$
We observe that
\begin{align*}
\frac{\int_{c_2\a}^\a d \|E_t(x^\dag-x_0)\|^2}{\int_0^\a d \|E_t(x^\dag-x_0)\|^2}
&=1-\frac{\int_0^{c_2\a} d \|E_t(x^\dag-x_0)\|^2}{\int_0^\a d\|E_t(x^\dag-x_0)\|^2}\\
&=1- \frac{\|E_{c_2\a}(x^\dag-x_0)\|^2}{\|E_\a(x^\dag-x_0)\|^2}.
\end{align*}
Therefore, for this scheme, Assumption \ref{ass:KN} is equivalent to the existence
of constants $0<c_2<1$, $0<\theta<1$ and $0<t_0<\|A\|$ such that
\begin{equation}\label{eq:KN-inequality}
\|E_{c_2\a}(x^\dag-x_0)\|\le \theta\|E_\a(x^\dag-x_0)\|, \quad \forall 0<\a\le t_0.
\end{equation}
}
\end{example}

\begin{example}
{\rm
For the asymptotical regularization we have $r_\a(t)=e^{-t/\a}$. Since $e^{-t/\a}\ge e^{-1}$ for
$c_2\a\le t\le \a$, it is easy to see that Assumption \ref{ass:KN} holds if (\ref{eq:KN-inequality})
is satisfied.
}
\end{example}

\begin{example}
{\rm
For the Landweber iteration with $\|A\|=1$, we have $r_\a(t)=(1-t)^{[1/\a]}$, where $[1/\a]$ denotes the largest
integer that is not greater than $1/\a$. Observing that for $0<t\le \a\le 1/2$ there holds
$(1-t)^{[1/\a]}\ge (1-t)^{1/\a} \ge (1-\a)^{1/\a} \ge 1/4$. Therefore, Assumption \ref{ass:KN}
 holds if (\ref{eq:KN-inequality}) is satisfied.
}
\end{example}

\subsection{Main result and discussion}
\label{sec:main-result}

The main result in this study is as follows.

\begin{thm}\label{thm:main}
Let assumptions \ref{ass:GWN}--\ref{ass:KN} hold. Let $\a_{RG}$ be chosen according to the
statistical RG-rule with $\kappa =\sqrt{8|\log
  (1/\delta)|/\cN(\alpha_{0})}$. Then there is a constant
$C$ such that
$$
\lr{\e{\norm{\xast - x_{\a_{RG}}^{\delta}}{}^{2}}}^{1/2} \leq C
\inf_{0 < \alpha \leq \alpha_0}
\set{\norm{\xa - \xast}{} + \frac{\delta (1+\sqrt{|\log (1/\delta)|})}{\Theta_{\rhon}(\alpha)}  }.
$$
  \end{thm}

The oracle inequality as established in Theorem~\ref{thm:main} allows
to state the error bound which is obtained under \emph{known} general source
condition and by an a priori parameter choice.
We recall some  notions.

\begin{de}[general source set]\label{de:gensource}
Given an index function $\psi$  that is continuous, non-negative, and non-decreasing
on $[0, \|A\|]$ with $\psi(0)=0$, the set
$$
H_{\psi} := \set{x\in X: \, x= \psi(A)v \mbox{ for some }\norm{v}{}\leq 1 },
$$
is called a general source set.
\end{de}

For solutions~$\xast$ which belong to some source set, the bias
$\norm{\xa - \xast}{}$ can be bounded under the assumption that the
chosen regularization has enough qualification, see e.g.\ \cite{MR2318806}

\begin{de}[qualification]\label{de:quali}
The regularization is said to have qualification~$\psi$ if there is a
constant $\gamma<\infty$ such that
$$
\abs{\ra(t)}\psi(t) \leq \gamma \psi(\alpha), \qquad \alpha>0.
$$
\end{de}

Notice  that $x^\dag-x_\a=r_\a(A)(x^\dag-x_0)$. If the regularization has
qualification $\psi$ and $x^\dag-x_0\in H_\psi$, then
$$
\|x^\dag-x_\a\| \le \gamma \psi(\a)\|v\| \le \gamma \psi(\a).
$$
By choosing $\a_\delta>0$ to be the root of the equation
$$
\Theta_{\rhon\psi}(\alpha):=\Theta_{\rhon}(t) \psi(t) = \delta\left(1+\sqrt{|\log (1/\delta)|}\right),
$$
we can use the the oracle inequality in Theorem \ref{thm:main} to obtain the following result.

\begin{cor}\label{cor:main}
 Let the assumptions~\ref{ass:GWN}--\ref{ass:KN} hold, and let $\a_{RG}$ be chosen according to the statistical
  RG-rule with $\kappa = \sqrt{8|\log (1/\delta)|/\cN(\alpha_{0})}$.
If the regularization has qualification~$\psi$ then
$$
\sup_{x^\dag-x_0 \in H_{\psi}} \lr{\e{\norm{\xast -
      x_{\a_{RG}}^{\delta}}{}^{2}}}^{1/2} \le
C \psi\left(\Theta_{\rhon\psi}^{-1} \left(\delta(1+\sqrt{|\log (1/\delta)|}) \right)\right).
$$
\end{cor}

Thus, up to a logarithmic factor,  the rate in Corollary \ref{cor:main} coincides with the
one from \cite[Theorem 1]{B/M2010}, which is known to be order optimal
in many cases.

We conclude this section with an outline of the proof of Theorem \ref{thm:main}.
The basic idea is to reduce the argument to the one for bounded
deterministic noise. The bound in Theorem~\ref{thm:main} uses the
effective dimension $\cN$, or more precisely the function~$\rhon$.
This function naturally appears when considering the
average performance of the noise under the weight $\sa^{1/2}(\opA)$ because
\begin{equation}
  \label{eq:sa-rhon}
 \lr{\e{\norm{\sa^{1/2}(\opA)\zeta}{}^{2}}}^{1/2}  =
 \lr{\tr{\sa(\opA)\opA}}^{1/2}  = \sqrt{\alpha\cN(\alpha)} =
\frac 1 {\rhon(\alpha)},\ \alpha>0.
\end{equation}
Therefore, we choose a tuning parameter~$\kappa$, as specified in Theorem~\ref{thm:main}, and define the set
\begin{equation}
  \label{eq:ztn}
  Z_{\kappa} := \set{\zeta:\, \norm{\sa^{1/2}(\opA)\zeta}{}\leq (1+\kappa)
    \frac {1}{\rhon(\alpha)},\ \hat{\a} \leq \alpha \in\Delta_{q}},
\end{equation}
where $\hat{\a}$ is the largest number in $\Delta_q$ satisfying
$$
\Theta_{\rhon}(\hat{\a}) \le \eta (1+\kappa) \delta.
$$
Let $Z_\kappa^c$ denote the complement of $Z_\kappa$ in $X$. Since $X=Z_\kappa\bigcup Z_\kappa^c$,
we can use the Cauchy-Schwarz inequality to derive that
 \begin{equation}
    \label{eq:error-bound-major}   \lr{\e{\norm{\xast -
          \xad}{}^{2}}}^{1/2}
 \leq \sup_{\zeta\in Z_{\kappa}}\norm{\xast - \xad}{} + \lr{\e{\norm{\xast -
          \xad}{}^{4}}}^{1/4} \lr{\prob{Z^{c}_{\kappa}}}^{1/4};
  \end{equation}
see \cite[Proposition 3]{B/M2010}. We will estimate the two terms on the right side of
(\ref{eq:error-bound-major}) with $\a=\a_{RG}$. Uniformly for $\zeta\in Z_{\kappa}$ the first term
on the right can be considered as error estimate under bounded
deterministic noise;  and we will show in Section~\ref{sec:aux} that it can be
bounded by the right hand side of the oracle inequality in
Theorem~\ref{thm:main}. This analysis may be of independent interest.
In Section~\ref{sec:gwn} we
will use some concentration inequality for Gaussian elements in
Hilbert space to show that the second
term on the right in~(\ref{eq:error-bound-major}) is negligible; this is enough for us to
complete the proof of Theorem~\ref{thm:main}

\section{Auxiliary results for bounded noise}
\label{sec:aux}

The situation for bounded deterministic noise which resembles
the Gaussian white noise case is regularization under some
specifically chosen weighted noise. We recall the function~$\sa$
from~(\ref{eq:sa}). As could be seen from the set~$Z_{\kappa}$
in~(\ref{eq:ztn}) the approriate  setup will be as follows.

\begin{ass}\label{ass:noise}
{\it There is a function $\a\to \delta(\a)>0$ defined on $(0, \infty)$ that is non-decreasing, while
$\a\to \delta(\a)/\sqrt{\a}$ is non-increasing such that the noise $\zeta$ obeys
  \begin{equation}
    \label{eq:ass-noise}
    \delta \|s_\a^{1/2}(A)\zeta\| \leq \delta(\alpha),\qquad \hat{\a}\le \a \in \Delta_q,
  \end{equation}
where $\hat{\a}\in \Delta_q$ is the largest parameter such that $\hat{\a}\le \eta \delta(\hat{\a})$ with
$\eta>0$ being a given small number.
}
\end{ass}

Because $\a\to \delta(\a)/\sqrt{\a}$ is non-increasing and $\a\to \sqrt{\a}$ is strictly increasing,
it is easy to see that $\hat{\a}$ is well-defined.

\begin{rem}
  \label{rem:general-noise}
{\rm
The setup in Assumption \ref{ass:noise} on noise covers a variety of cases which have been
subsumed under the notion of \emph{general noise assumptions}, we refer
to \cite{0266-5611-27-3-035016,B/M2010b}.
Specifically, let us consider the following situation. Suppose that
the noise $\zeta$ allows for a noise bound for some parameter $\mu$ with
\begin{equation}
  \label{eq:zeta-bound}
  \norm{\opA^{-\mu}\zeta}{}\leq 1.
\end{equation}
In this case we can bound
$$
\delta \norm{\sa^{1/2}(\opA)\zeta}{} \leq
\delta \norm{\sa^{1/2}(\opA)\opA^{\mu}}{}\norm{\opA^{-\mu} \zeta}{}\leq
\norm{s_\a^{1/2}(\opA)\opA^{\mu}}{} \delta.
$$
It is easily verified that the operator
norms $\norm{\sa^{1/2}(\opA)\opA^{\mu}}{}$ are uniformly bounded for $\a>0$ if and only if
$0\leq \mu \leq 1/2$. In this range we easily obtain that
$$
\norm{\sa^{1/2}(\opA)\opA^{\mu}}{} \leq \alpha^{\mu},\quad
\alpha >0.
$$
The two limiting cases are $\mu=0$, where we assume $\norm{\zeta}{}=
\norm{T^{\ast}\xi}{}\leq 1$ which corresponds to \emph{large noise},
and $\mu=1/2$, where we assume $\norm{\opA^{-1/2}\zeta}{}= \norm{\xi}{}{\leq 1}$ which
corresponds to the usual noise assumption in linear inverse problems in Hilbert
spaces. In any of the cases $0\leq \mu \leq 1/2$ we get a bounding function
$\delta(\alpha) = \delta \alpha^{\mu}$, which obeys the requirements
made in Assumption~\ref{ass:noise}.
}
\end{rem}

Let $\hat{\a}\in \Delta_p$ be defined as in Assumption
\ref{ass:noise}, i.e. $\hat{\a}\in \Delta_q$ is the largest parameter such that
$\hat{\a}\le \eta \delta(\hat{\a})$.

\begin{de}[RG-rule]\label{def:RG}
Given $\tau>1$ and $\eta>0$,   we define $\aast\in \Delta_{q}$ to be the largest
parameter such that
  \begin{equation}
    \label{eq:rg-bound}
\a_*\ge \hat{\a} \quad \mbox{and} \quad \|s_{\a_*}(A)(A x_{\a_*}^\delta - \zd)\| \leq  \tau \delta(\a_*);
  \end{equation}
if such $\a_*$ does not exist, we define $\a_*:=\hat{\a}$.
\end{de}

We notice that the norm in the above criterion can be rewritten as
$$
 \norm{\sa(\opA) (\opA \xad - \zd)}{}
=  \norm{\sa(\opA)\ra(\opA)(\opA \xn - \zd)}{}.
$$

\subsection{Properties of the RG-rule}
\label{sec:RG-prop}


We give some technical consequences of the stopping criterion which will be used later.

\begin{lem}\label{lem:case1}
  Let $\a\in\Delta_{q}$ be any parameter such that $\a>\a_*$. Then there holds
$$
\frac{\delta(\a)}{\a} \leq \frac 1 {\tau -1} \norm{\xa
  - \xast}{}.
$$
\end{lem}

\begin{proof}
Since $\a>\a_*$, by the definition of $\a_*$ we must have
$$
\tau \delta(\a) \le \|s_\a(A) r_\a(A) (A x_0-z^\delta)\|.
$$
Therefore, it follows from Assumption \ref{ass:noise} that
\begin{align*}
    \tau \delta(\alpha) &  \leq   \norm{\sa(\opA) \ra(\opA)(z - \zd)}{}
+  \norm{\sa(\opA)\ra(\opA)(\opA \xn - z)}{} \\
& \leq \norm{\sa^{1/2}(\opA)\ra(\opA)}{} \delta(\alpha) +
 \norm{\sa(\opA) A}{} \norm{ \xa- \xast}{}.
  \end{align*}
Since $0\leq \sa^{1/2}(t) \ra(t)\leq 1$ and $0\le s_\a(t) t\le \a$, we have
$\norm{\sa^{1/2}(\opA)\ra(\opA)}{}\leq 1$ and $\norm{\sa(\opA)\opA}{} \leq \alpha$.
Consequently
$$
(\tau - 1) \delta(\alpha) \leq {\alpha} \norm{\xa - \xast}{},
$$
which gives the estimate.
\end{proof}

\begin{lem}\label{lem:case2}
  Let the parameter~$\aast$ be chosen by the RG-rule in Definition~\ref{def:RG}. Then
$$
\norm{\saa(\opA)\raa(\opA)(\opA \xn - z)}{} \leq \gamma_0  \delta(\aast),
$$
where $\gamma_0:=\max\{1+ \tau, \eta\|x^\dag-x_0\|\}$.
\end{lem}

\begin{proof}
If $\a_*=\hat{\a}$, then it follows from the definition of $\hat{\a}$ that $\a_*\le \eta \delta(\a_*)$.
Consequently
\begin{align*}
\|s_{\a_*}(A) r_{\a_*}(A) (A x_0-z)\| &= \|s_{\a_*}(A) r_{\a_*}(A) A (x^\dag-x_0)\|\\
& \le \a_* \|x^\dag-x_0\| \le \eta \|x^\dag-x_0\| \delta(\a_*).
\end{align*}
Otherwise we have that $\a_*>\hat{\a}$. Then by the definition of $\a_*$ we have
\begin{align*}
&\norm{\saa(\opA)\raa(\opA)(\opA \xn - z)}{} \\
&\leq \norm{\saa(\opA) \raa(\opA)(z - \zd)}{}
+    \norm{\saa(\opA)\raa(\opA)(\opA \xn - \zd)}{} \\
&\leq   \norm{s_{\aast}^{1/2}(\opA)\raa(\opA)}{} \delta(\aast) + \tau
\delta(\aast) \leq (1+\tau) \delta(\aast),
  \end{align*}
and the proof is complete.
\end{proof}

\subsection{Auxiliary inequalities: The impact of
  Assumption~\ref{ass:KN}}
\label{sec:ineq}

The following inequalities may be of general interest. The first one
goes back to~\cite{MR1680887,MR1702595}, see also~\cite[Lemma 2.4]{MR2684361}.

\begin{lem}\label{lem:qinian-ineq}
  For $0 < \alpha \leq \beta$ we have
$$
\norm{\xb - \xa}{} \leq \frac{1 + \gamma_{\ast}}{{\sqrt \alpha}}
\norm{\opA^{1/2}\sbeta^{1/2}(\opA)\rb(\opA)(\xast - \xn)}{} .
$$
\end{lem}

\begin{proof}
We first notice that $\xb - \xa = (\rb(\opA) - \ra(\opA))(\xast - \xn)$.
  The bound established in~(\ref{eq:lem23}) yields that
  \begin{align*}
\norm{\xb - \xa}{}  &= \norm{(\rb(\opA) - \ra(\opA))(\xast - \xn)}{}\\
&\leq (1 + \gamma_{\ast})\norm{\opA \lr{\alpha + \opA}^{ -1}
  \rb(\opA)(\xast - \xn)}{}  \\
 & = \frac{1 + \gamma_{\ast}}{\alpha}
 \norm{\opA\sa(\opA)\rb(\opA)(\xast - x_{0})}{}.
  \end{align*}
We may write
$$
\opA\sa(\opA) = \opA^{1/2}\sa^{1/2}(\opA)
\frac{1}{\sbeta^{1/2}}(\opA) \sa^{1/2}(\opA) \opA^{1/2} \sbeta^{1/2}(\opA).
$$
Observing that $0\le s_\a(t) t^{1/2}\le \sqrt{\a}$ and $\sa(t) \leq \sbeta(t)$ for $t\ge 0$, we have that
$\norm{\sa^{1/2}(\opA)\opA^{1/2}}{}\leq \sqrt\alpha$ and
$\norm{\frac{1}{\sbeta^{1/2}}(\opA) \sa^{1/2}(\opA)}{}\leq 1$. Therefore
$$
\norm{\opA \sa(\opA)
  \rb(\opA)(\xast - \xn)}{} \leq \sqrt \alpha \norm{ \opA^{1/2}
  \sbeta^{1/2}(\opA)\rb(\opA)(\xast - \xn)}{},
$$
which allows to  complete the proof.
\end{proof}

The bound from Lemma \ref{lem:qinian-ineq} does not suffice, and we need
the following strengthening, where Assumption~\ref{ass:KN} is crucial.

\begin{lem}\label{lem:kn}
Suppose that Assumption~\ref{ass:KN} holds true. Then there is a constant $C<\infty$ such that
for $0 < \alpha \leq \alpha_0$  there holds
$$
\norm{\opA^{1/2}\sa^{1/2}(\opA)\ra(\opA)(\xast -
  \xn)}{}
 \leq \frac{C}{\sqrt\alpha} \norm{\sa(\opA)\ra(\opA)\opA(\xast - \xn)}{}.
$$
\end{lem}

\begin{proof}
We use spectral calculus to write
\begin{align*}
\norm{\opA^{1/2}\sa^{1/2}(\opA)\ra(\opA)(\xast -
  \xn)}{}^{2} =I_1(\alpha) +I_2(\alpha),
\end{align*}
where
\begin{align*}
I_1(\alpha)&:= \int_{0}^{ \alpha} t \sa(t)  \ra^{2}(t) \;
d\norm{E_{t}(\xast - \xn)}{}^{2} \\
I_2(\alpha)&:= \int_{\alpha}^{ \infty} t  \sa(t)  \ra^{2}(t) \;
d\norm{E_{t}(\xast - \xn)}{}^{2}.
\end{align*}
We first bound $I_2$.
For $t\geq \alpha$ we have that $\alpha(t+\alpha) \leq 2 \alpha t$,
thus $1 \leq \tfrac 2 \alpha t \sa(t)$, yielding
\begin{align*}
I_2(\alpha) \leq
\frac 2 {\alpha}\int_{\alpha}^{ \infty} t^{2} \sa^{2}(t) \ra^{2}(t) \;
d\norm{E_{t}(\xast - \xn)}{}^{2} \leq \frac 2 {\alpha} \norm{\sa(\opA)\ra(\opA)\opA(\xast - \xn)}{}^{2}.
\end{align*}
To estimate~$I_1(\alpha)$ we will use Assumption~\ref{ass:KN}.
 We will consider two cases:
$0<\alpha \le t_0$ and $t_0 <\alpha \le \alpha_0$.

When $0<\alpha\le t_0$, we use Assumption~\ref{ass:KN} to
obtain from $t\sa(t)\leq \alpha$ that
\begin{align*}
I_1(\alpha)   \leq \alpha  \int_{0}^{ \alpha}
\;d\norm{E_{t}(\xast - \xn)}{}^{2}  \leq c_1^{2}\alpha \int_{c_2 \alpha}^{\infty} r_\alpha^2 (t)
\;d\norm{E_{t}(\xast - \xn)}{}^{2}.
\end{align*}
Since $t/(t + \alpha)\geq c_2/(1+c_2)$ for $t\ge c_2 \alpha$, we further obtain
\begin{align*}
I_1(\alpha) & \leq  \frac{c_1^{2}(1+c_2)^2}{c_2^2\alpha}   \int_{c_2 \alpha}^{\infty}
\frac{\alpha^{2}t^{2}}{(t + \alpha)^{2}}\ra^{2}(t) \;d\norm{E_{t}(\xast - \xn)}{}^{2}\\
&= \frac{c_1^{2}(1+c_2)^2}{c_2^2\alpha} \int_{c_2 \alpha}^{\infty} \sa^{2}(t)\ra^{2}(t) t^{2}
\;d\norm{E_{t}(\xast - \xn)}{}^{2}\\
&\leq  \frac{c_1^{2} (1+c_2)^2}{c_2^2\alpha} \norm{\sa(\opA)\ra(\opA)\opA(\xast - \xn)}{}^{2}.
\end{align*}

Now we consider the case $t_0 <\alpha \leq\alpha_0$. We write
$I_1(\alpha)=I_1^{(1)}(\alpha) + I_1^{(2)}(\alpha)$, where
\begin{align*}
I_1^{(1)}(\alpha) & :=\int_0^{t_0} t  s_\alpha (t)
r_\alpha^2(t) d \|E_t (x^\dag-x_0)\|^2,\\
I_1^{(2)}(\alpha) & := \int_{t_0}^\alpha  t s_\alpha (t)
r_\alpha^2(t) d \|E_t (x^\dag-x_0)\|^2.
\end{align*}
We can bound, by using  Assumption~\ref{ass:KN},  the term~$I_1^{(1)}(\alpha)$ as
$$
I_1^{(1)}(\alpha) \le c_1^{2} \alpha \int_{c_2 t_0}^\infty r_{t_0}^2 (t) d \|E_t(x^\dag-x_0)\|^2.
$$
Since $t_0\le \alpha$ implies $r_{t_0}(t) \le r_\alpha(t)$, we have
$$
I_1^{(1)}(\alpha) \le c_1^{2} \alpha \int_{c_2 t_0}^\infty r_\alpha^2 (t) d \|E_t(x^\dag-x_0)\|^2.
$$
Observing that for $t\ge c_2 t_0$ there holds $\frac{t}{t+\alpha} \ge \frac{t}{t+\alpha_0}
\ge \frac{c_2 t_0}{c_2 t_0 +\alpha_0}$, we further obtain
\begin{align*}
I_1^{(1)}(\alpha) & \le \left(\frac{c_2 t_0 +\alpha_0}{c_2
    t_0}\right)^2 \frac{c_1^{2}}{ \alpha}
\int_{c_2 t_0}^\infty \frac{t^2\alpha^{2}}{(t+\alpha)^2} r_\alpha^2(t) d \|E_t(x^\dag-x_0)\|^2\\
& = \frac{c_1^{2}}{ \alpha}\left(\frac{c_2 t_0 +\alpha_0}{c_2 t_0}\right)^2
\int_{c_2 t_0}^\infty s_\alpha^2(t) r_\alpha^2(t) t^2 d \|E_t(x^\dag-x_0)\|^2\\
&\le \frac{c_1^{2}}{ \alpha}\left(\frac{c_1 t_0 +\alpha_0}{c_2 t_0}\right)^2
\|s_\alpha(A) r_\alpha(A) A (x^\dag-x_0)\|^2.
\end{align*}

To bound $I_{1}^{(2)}$, we observe that for $t_0 \le t\le \alpha$ there holds
$  1\leq \frac{\alpha_0+t_0}{t_0 \alpha}t\sa(t)$.
Consequently
\begin{align*}
I_1^{(2)}(\alpha) &\le \frac{\alpha_0+ t_0}{t_0 \alpha}
\int_{t_0}^\alpha t^2
\sa^2(t) r_\alpha^2(t) d \|E_t(x^\dag-x_0)\|^2 \\
& \le \frac{\alpha_0+ t_0}{t_0 \alpha} \|s_\alpha(A)
r_\alpha(A) A (x^\dag-x_0)\|^2.
\end{align*}
Combining the above estimates we therefore obtain the desired bound
with
$C=\lr{
2+ \frac{c_1^{2}(1+c_2)^2}{c_2^2 } +\frac{\alpha_0+ t_0}{t_0}
  + c_1^{2} \lr{\frac{c_2 t_0 +\alpha_0}{c_2 t_0}}^2}^{1/2}$.
\end{proof}

We summarize the results from Lemma \ref{lem:qinian-ineq} and Lemma
\ref{lem:kn} as follows.

\begin{cor}
  \label{cor:xb-xa-bound}
Let Assumption \ref{ass:KN} hold. Then there is a constant $C<\infty$ such that
for all $0<\a\le \beta\le \a_0$ there holds
$$
\|\xb - \xa\| \leq \frac{C}{\sqrt{\a \beta}}
\|s_\beta(A)\rb(A) A(\xast - \xn)\|.
$$

\end{cor}

\subsection{Deterministic oracle inequality}
\label{sec:oracle}

In this section we state the main auxiliary result for bounded
deterministic noise, as this seems to be of independent interest.

\begin{thm}\label{thm:oracle-det}
Let the assumptions \ref{ass:KN} and \ref{ass:noise} hold, and let
the parameter $\alpha_{\ast}$ be chosen by the RG-rule starting with $\alpha_{0}$.
Then there holds the oracle inequality, i.e. there is a constant $C$ such that
  \begin{equation}
    \label{eq:oracle}
    \norm{x_{\alpha_{\ast}}^{\delta} - \xast}{} \leq C
    \inf_{0< \alpha\leq \alpha_0}\set{\norm{\xa - \xast}{} +
      \frac{\delta(\alpha)}{\alpha}}.
  \end{equation}
\end{thm}

\begin{proof}
We first derive some preparatory results. Observing that
$x^\dag-x_\a =r_\a(A) (x^\dag-x_0)$, we have from (\ref{eq:ra-monotone}) that
\begin{equation}\label{eq:bais-monotone}
\|x^\dag-x_\a\|\le \|x^\dag-x_\beta\|, \qquad \forall 0<\a\le \beta.
\end{equation}
By the conditions on $g_\a$ we have
$$
\frac{\ga(t)}{\sa^{1/2}(t)} =\frac{1}{\sqrt{\a}} \sqrt{\ga(t)} \sqrt{\ga(t) (\alpha+t)}
\le \frac{\sqrt{\gamma_*(1+\gamma_*)}}{\alpha}.
$$
Therefore, with $c_*=\sqrt{\gamma_*(1+\gamma_*)}$ we have
\begin{align}
  \norm{\xast - \xad}{} &  \leq \norm{\xast - \xa}{} +
  \norm{\ga(\opA)\lr{\frac 1
      {\sa^{1/2}}}(\opA)\brac{\sa^{1/2}(\opA)(z -
      \zd)}}{}\notag\\
& \leq  \norm{\xast - \xa}{} +
\frac{c_*}{\alpha}
\norm{\sa^{1/2}(\opA) \zeta}{} \delta.  \label{eq:point-wise-bound}
\end{align}
It then follows from Assumption~\ref{ass:noise} that
\begin{align}
\|x^\dag-x_\a^\delta\| \leq  \norm{\xast - \xa}{} +c_* \frac{\delta(\alpha)}{\alpha}.\label{eq:point-error}
\end{align}

Next we will prove the oracle inequality in two steps. We first restrict the oracle
bound to $\alpha\in\Delta_{q}$, and we show that
\begin{equation}
   \label{eq:oracle1}
    \norm{x_{\alpha_{\ast}}^{\delta} - \xast}{} \leq C
    \inf_{\a\in \Delta_q} \set{\norm{\xa - \xast}{} +
      \frac{\delta(\alpha)}{\alpha}}.
\end{equation}
In this case we shall distinguish the cases $\alpha > \aast$ and $\alpha \leq
\aast$, respectively.

\begin{description}

\item[Case $\alpha > \aast$] We first have from (\ref{eq:point-error}), (\ref{eq:bais-monotone})
and the monotonicity of $\a\to \delta(\a)$ that
\begin{align*}
\|x_{\a_*}^\delta -x^\dag\|  \le \|x_{\a_*}-x^\dag\| + c_* \frac{\delta(\a_*)}{\a_*} \le \|x_\a-x^\dag\| +c_* \frac{\delta(\a_*/q)}{\a_*}.
\end{align*}
Since $\a, \a_*\in \Delta_q$, we have $\a_*/q\in \Delta_q$ and $\a\ge \a_*/q>\a_*$. Then we can conclude,
by using Lemma \ref{lem:case1} and (\ref{eq:bais-monotone}), that
\begin{align*}
\|x_{\a_*}^\delta - x^\dag\| \le  \|x_\a - x^\dag\| + \frac{c_*}{q(\tau-1)} \|x_{\a_*/q}  - x^\dag\| \leq \lr{1 + \frac{c_*}{q(\tau -1)}} \|x_\a - x^\dag\|.
\end{align*}

\item[Case $\alpha \leq  \aast$]
We actually use Assumption~\ref{ass:KN} and its consequences.
 Based on Corollary \ref{cor:xb-xa-bound} and Lemmas \ref{lem:case2} we
 conclude in this case that there is a constant $C<\infty$ with
\begin{align*}
 \|x_{\a_*} - x^\dag\| & \leq \|x_\a - x^\dag\|
+ \frac{C}{\sqrt{\a\a_*}} \|s_{\a_*} (A) r_{\a_*} (A)(A \xn - z)\|\\
& \leq \|x_\a-x^\dag\|  + C \gamma_0 \frac{\delta(\a_*)}{\sqrt{\a\a_*}}.
\end{align*}
Consequently, we deduce, using the bound~(\ref{eq:point-error}) and
that $\a\to \delta(\a)/\sqrt{\a}$ is non-increasing, that
\begin{align*}
\|x_{\a_*}^{\delta} - \xast\| & \leq   \|x_{\a_*} - \xast\| + c_* \frac{\delta(\a_*)}{\a_*}
 \leq  \|\xa - \xast\| + C \gamma_0 \frac{\delta(\a_*)}{\sqrt{\a\a_*}} +
c_* \frac{\delta(\a_*)}{\a_*} \\
& \leq \left(C\gamma_0 +c_*\right)\lr{\|\xa - \xast\| + \frac{\delta(\a)}{\a} }.
\end{align*}
\end{description}

Finally, we show the oracle inequality in its full generality.  To this end, let
$0<\a\le \a_0$ be any number. Then there is $j\in\nat$ such that $\a_j < \a \le \a_j/q$.
By using (\ref{eq:bais-monotone}), the fact that $\a\to \delta(\alpha)$ is increasing, and the
fact that $\alpha\to \delta(\alpha)/\alpha$ is decreasing, we obtain
\begin{align*}
\|x_\alpha-x^\dag\| + \frac{\delta(\alpha)}{\alpha}
&\ge \|x_{\alpha_j}-x^\dag\| +\frac{\delta(\alpha_j/q)}{\alpha_j/q}
\ge  q\left(\|x_{\alpha_j}-x^\dag\| +\frac{\delta(\alpha_j)}{\alpha_j} \right)\\
& \ge q \inf_{\beta\in \Delta_q} \left\{ \|x_\beta-x^\dag\| + \frac{\delta(\beta)}{\beta} \right\}.
\end{align*}
Since $0<\a\le \a_0$ is arbitrary, we obtain
$$
\inf_{0<\a\le \a_0} \left\{\|x_\a-x^\dag\| +\frac{\delta(\a)}{\a}\right\}
\ge q \inf_{\a\in \Delta_q} \left\{ \|x_\a-x^\dag\| +\frac{\delta(\a)}{\a}\right\}.
$$
The proof is therefore complete.
\end{proof}

\subsection{Discussion}
\label{sec:discussion}
(a) From Lemma \ref{lem:case1} and (\ref{eq:bais-monotone}) it follows that
$$
\frac{\delta(\a_*/q)}{\a_*/q}\le \frac{1}{\tau-1} \|x_{\a_*/q}-x^\dag\| \le \frac{1}{\tau-1}\|x_0-x^\dag\|.
$$
Since $\a\to \delta(\a)$ is non-decreasing, we obtain
$$
\frac{\delta(\a_*)}{\a_*}\le \frac{q}{\tau-1} \|x_0-x^\dag\|.
$$
If in the definition of $\hat{\a}$ we take $0<\eta<\frac{\tau-1}{q\|x_0-x^\dag\|}$, then we always have
$\a_*>\hat{\a}$. Therefore, the RG rule in Definition \ref{def:RG} simply reduces to the form: {\it $\a_*$ is the
largest parameter in $\Delta_q$ such that
$$
\|s_{\a_*}(A) (A x_{\a_*}^\delta-z^\delta) \| \le \tau \delta(\a_*).
$$
}
The oracle inequality in Theorem \ref{thm:oracle-det} still holds for this simplified parameter choice rule.

(b) The oracle inequality established in Theorem \ref{thm:oracle-det} can be used to
yield error bounds when the solution $\xast$ has smoothness given in
terms of \emph{general source conditions}, i.e.,\ if $\xast-x_0$ belongs
to some source set introduced in Definition \ref{de:gensource}. To see this, we assume
that the regularization has qualification $\psi$ as in Definition \ref{de:quali}
and $x^\dag-x_0\in H_\psi$. We also assume, as introduced in Remark \ref{rem:general-noise},
that the noise can be bounded as $\|A^{-\mu}\zeta\|\leq 1$, which results in
$\delta(\alpha) = \delta \alpha^{\mu}$ for $0\leq \mu \leq 1/2$. Then, for the parameter
$\a_*$, determined by the RG rule in Definition \ref{def:RG}, it follows from
Theorem \ref{thm:oracle-det} that
$$
\|x_{\a_*}^\delta-x^\dag\| \le C \inf_{0<\a\le \a_0} \left\{\psi(\a)
+ \frac{\delta}{\alpha^{1-\mu}}\right\}.
$$
Associated to the smoothness $\psi$, let $\Theta_{\mu,\psi}(t) := t^{1-\mu} \psi(t)$,
$t>0$, which is a strictly increasing function. Given $\delta>0$ we assign
$\alpha_{\delta}>0$ such that $\Theta_{\mu,\psi}(\alpha_{\delta})=\delta$. Then we can conclude that
$$
\|x_{\a_*}^\delta-x^\dag\| \le C \left\{\psi(\a_\delta) + \frac{\delta}{\a_\delta^{1-\mu}}\right\}
\le 2C  \psi(\Theta_{\mu,\psi}^{-1}(\delta)),
$$
which was shown to be order optimal for $\xast$ with the above smoothness
in \cite[Theorem 4]{0266-5611-27-3-035016}. Thus, the present results
cover part of the analysis carried out in \cite{0266-5611-27-3-035016}; it extends the stopping
criteria studied there to the RG-rule, and hence this relates
to \cite{M/T2010b}. However, the above approach is limited. First, the case of \emph{small
noise}, i.e.,\ when $-1/2 < \mu \leq 0$ cannot be covered. Secondly,
the oracle inequality is seen to hold only for those solutions~$\xast$ satisfying
Assumption~\ref{ass:KN}.

\section{Proof of the main result}
\label{sec:gwn}

The proof of Theorem \ref{thm:main} will be carried out in several steps,  similar to the
one in the recent studies \cite{B/M2010,Lu/Mathe2012}. Our starting point is the inequality
(\ref{eq:error-bound-major}). Recall that $Z_{\kappa}$ is the set defined by (\ref{eq:ztn}), i.e.
$Z_{\kappa}\subset X$ consists of those realizations of the noise $\zeta$ obeying Assumption~\ref{ass:noise}
along the sequence $\alpha_{0},\dots,\hat{\a}$ with
\begin{equation}
  \label{eq:delta-gwn}
  \delta(\alpha):= (1+\kappa) \frac\delta{\rhon(\alpha)}
  ,\quad \alpha>0,
\end{equation}
where $\hat{\a}$ is the largest number in $\Delta_q$ satisfying
\begin{equation}\label{alpha-hat}
\Theta_{\rhon}(\hat{\a}) \le \eta (1+\kappa) \delta,
\end{equation}
According to the definition of $\a_{RG}$ we have $\a_{RG}\ge \hat{\a}$.

In order to estimate the first term on the right of (\ref{eq:error-bound-major}) with $\a:=\a_{RG}$,
we observe that when $\zeta\in Z_\kappa$, the parameter $\a_*$ determined by the RG rule in
Definition \ref{def:RG} with $\delta(\a)$ given by (\ref{eq:delta-gwn}) is equal to the parameter $\a_{RG}$
determined by the statistical RG rule in Definition \ref{de:VRG}. Therefore we may use Theorem \ref{thm:oracle-det}
to conclude
\begin{equation}\label{oracle-det2}
\sup_{\zeta\in Z_\kappa} \|x^\dag-x_{\a_{RG}}^\delta\| \le C \inf_{0<\a\le \a_0}
\left\{ \|x^\dag-x_\a\| +\frac{(1+\kappa) \delta}{\Theta_{\rhon}(\a)}\right\}.
\end{equation}

In the following we will estimate the second term on the right side of (\ref{eq:error-bound-major}) with $\a=\a_{RG}$.
We need some auxiliary results.

\begin{lem}\label{lem:number-steps}
Let Assumptions \ref{ass:power-type} hold. Let $\hat{\a}=\alpha_{0}q^{\hat{n}} \in\Delta_{q}$
be the largest parameter satisfying (\ref{alpha-hat}). Then there is a constant $C$ such that
$$
\hat{n} \leq C\left(1+|\log(1/\delta)|\right).
$$
\end{lem}

\begin{proof}
Since $\|A\|>0$ is the first eigenvalue of $A$, it follows from the definition
of ${\mathcal N}(\a)$ that
$$
{\mathcal N}(\a) = \tr{(\a I +A)^{-1} A} \ge \frac{\|A\|}{\a+\|A\|} \ge \frac{\|A\|}{\a_0+\|A\|},
\quad 0<\a\le \a_0.
$$
Therefore, with $C_0: =\sqrt{(\a_0+\|A\|)/\|A\|}$, we obtain
$$
 \Theta_{\rhon} (\a) =\sqrt{\frac{\a}{\cN(\a)}} \le C_0 \a^{1/2}, \quad 0<\a\le \a_0.
$$
According to the definition of $\hat{\a}$ we have
$$
\Theta_{\rhon}(\hat{\a}/q) >\eta (1+\kappa) \delta \ge \eta \delta.
$$
Consequently $C_0 (\hat{\a}/q)^{1/2} \ge \eta \delta$ which implies the result.
\end{proof}

We shall also use some prerequisites from Gaussian random elements in
Banach spaces, and we recall the following results from \cite[Lemma~3.1
\& Corollary~3..2]{MR1102015}.

\begin{lem}\label{lem:LT3.1}
  Let $\Xi$ be any Gaussian element in some Banach space. Then
$$
\prob{\norm{\Xi}{} > \e{\norm{\Xi}{}} + b} \leq e^{- \frac{b^2}{2
v^2}},
$$
with $v^{2}:= \sup_{\norm{w}{}\leq 1}\e{\scalar{\Xi}{w}}^{2}$. Moreover,
for each $p>1$ there is a constant $C_{p}$ such that
$$
\e{\norm{\Xi}{}^{p}}^{1/p} \leq C_{p} \e{\norm{\Xi}{}}.
$$
\end{lem}

We apply Lemma~\ref{lem:LT3.1} to $\Xi:= \sa^{1/2}(\opA)\zeta \sa^{1/2}(\opA)T^{\ast}\xi$.
For fixed $\alpha\in\Delta_{q}$ we denote
$$
  Z_{\kappa,\alpha} := \set{\zeta: \, \norm{\sa^{1/2}(\opA)\zeta}{}\leq (1+\kappa)
    \frac {1}{\rhon(\alpha)}}.
$$

\begin{cor}\label{cor:zka-bound}
For each $0 < \alpha\leq \alpha_{0}$ there holds
$$
\prob{Z_{\kappa,\alpha}^c} \leq e^{- \frac{\kappa^{2}\cN(\alpha)} 2}
\qquad \mbox{and} \qquad
  \lr{\e{\norm{\sa^{1/2}(\opA) \zeta}{}^{4}}}^{1/4} \leq C_{4}  \frac{1}{\rhon(\alpha)}.
$$
\end{cor}

\begin{proof}
We first estimate ${\mathbb P}[Z_{\kappa,\a}^c]$.  The expected norm
of $\Xi$ can be bounded, cf.~(\ref{eq:sa-rhon}), as
\begin{align}
  \e{\norm{ \sa^{1/2}(\opA)\zeta}{}}
\leq  \lr{\e{\norm{\sa^{1/2}(\opA)\zeta}{}^{2}}}^{1/2} = \frac{1}{\rhon(\alpha)} \label{eq:enorm-bound}.
\end{align}
For any $w\in X$ with $\|w\|\le 1$, the weak second moments can be bounded from above by
$$
\e{\scalar{\Xi}{w}}^{2} =\e{\scalar{\xi}{ T \sa^{1/2}(\opA)w}}^{2}
= \norm{ T \sa^{1/2}(\opA)w}{}^{2} \leq \norm{ T
  \sa^{1/2}(\opA)}{}^{2} \leq \alpha.
$$
Thus we may apply Lemma \ref{lem:LT3.1} with $b:= \kappa
/\rhon(\a)$ to conclude that
$$
\prob{Z_{\kappa,\alpha}^c} \leq e^{- \frac{\kappa^{2}}{2\alpha
    \rhon^{2}(\alpha)}} = e^{- \frac{\kappa^{2}\cN(\alpha)}{2}},
$$
which completes the proof of the first assertion. The second one is a
consequence of~(\ref{eq:enorm-bound}) and Lemma~\ref{lem:LT3.1}.
\end{proof}

Finally we turn to the proof of the main result.

{\it Proof of Theorem \ref{thm:main}.}
We will use (\ref{eq:error-bound-major}) with $\a=\a_{RG}$. The first
term on the right has been estimated
in (\ref{oracle-det2}). By using Lemma \ref{lem:number-steps} and  Corollary \ref{cor:zka-bound}
we obtain from $Z_\kappa^c =\bigcup_{\hat{\a}\le \a\in \Delta_q} Z_{\kappa,\a}^c$ that
 \begin{align*}
 \prob{Z_{\kappa}^{c}}  \leq  (\hat{n}+1) \sup_{ \hat{\a}\leq \a \in\Delta_{q}} \prob{Z_{\kappa,\alpha}^{c}}
 \leq C\left(1+|\log(1/\delta)|\right) e^{- \frac{\kappa^{2}\cN(\alpha_{0})}{2}}.
 \end{align*}
For~$\kappa = \sqrt{8\abs{\log(1/\delta)}/\cN(\alpha_{0})}$ this yields
\begin{equation}\label{prob-z-c}
\prob{Z_{\kappa}^{c}} \le C\left(1+|\log(1/\delta)|\right) \delta^{4}
\le C \left(1+\sqrt{|\log(1/\delta)|}\right)^4 \delta^4.
\end{equation}

It remains to establish a bound for
$\e{\norm{\xast-x_{\a_{RG}}^{\delta}}{}^{4}}$. We emphasize that the
random element~$x_{\a_{RG}}^{\delta}$ is no longer Gaussian in general,
since the parameter $\alpha_{RG}$ depends on the
data~$\zeta$. Hence we cannot apply Lemma~\ref{lem:LT3.1} directly.
Therefore we will use the error bound~(\ref{eq:point-wise-bound}) which is
valid for every $\zeta$. By using the facts that $\a_{RG}\ge\hat{\a}$
and that the function $\alpha \mapsto \sa^{1/2}(t)/\alpha$ is
decreasing for each $t\ge 0$, we obtain
\begin{align*}
\|x^\dag-x_{\a_{RG}}^\delta\|
& \leq \|x^\dag-x_{\a_{RG}}\| + c_{\ast} \delta
\frac{\|s_{\a_{RG}}^{1/2}(A)\zeta\|}{\a_{RG}} \le \|x^\dag-x_0\| + c_*
\delta \frac{\|s_{\hat{\a}}^{1/2} (A) \zeta\|}{\hat{\a}}.
\end{align*}
Since $\hat\alpha$ is deterministic, the element~$s_{\hat{\a}}^{1/2}
(A) \zeta$ is Gaussian. Thus we may use the bound on the fourth moment
of $\|s_{\hat{\a}}^{1/2}(A)\zeta\|$ given in Corollary~\ref{cor:zka-bound} to obtain
$$
\lr{\e{\norm{x -x_{\a_{RG}}^{\delta}}{}^{4}}}^{1/4} \leq  \norm{\xast -
  x_{0}}{} + c_{\ast}C_{4}\frac{\delta}{\Theta_{\rhon}(\hat{\a})}.
$$
Since the function $\a\to \rhon(\a)$ is decreasing, it is easy to obtain that $\Theta_{\rhon}(\hat{\a})
\ge q \Theta_{\rhon}(\hat{\a}/q)$. By the definition of $\hat{\a}$ we then obtain
$\Theta_{\rhon}(\hat{\a}) \ge q \eta(1+\kappa) \delta \ge q\eta \delta$. Consequently
\begin{equation}\label{fourth-moment}
\lr{\e{\norm{x -x_{\a_{RG}}^{\delta}}{}^{4}}}^{1/4} \leq \|x^\dag-x_0\| + \frac{c_* C_4}{q \eta}.
\end{equation}
Combining the estimates (\ref{oracle-det2}), (\ref{prob-z-c}) and (\ref{fourth-moment}) with
(\ref{eq:error-bound-major}) and we use that $0<\a\le \a_0$ yields $\Theta_{\rhon} (\a_0)/\Theta_{\rhon}(\a)\ge 1$. We can conclude that
$$
\lr{\e{\norm{\xast - x_{\a_{RG}}^{\delta}}{}^{2}}}^{1/2} \leq C
\inf_{0<\alpha \leq \a_0}
\set{\norm{\xa - \xast}{} +
  \frac{\delta(1+\kappa)}{\Theta_{\rhon}(\alpha)}}.
$$
The proof is therefore complete. $\Box$

\def\cprime{$'$}

\end{document}